\documentclass{article}
\usepackage{amsmath, amsthm, url, amssymb, authblk}

\title{Divisibility of generalized Mersenne numbers}
\author{Alex Chan}
\affil{alexandr13s98765@gmail.com}
\date{2025}

\theoremstyle{plain}
\newtheorem*{theorem}{Theorem}

\theoremstyle{remark}
\newtheorem*{remark}{Remark}
\newtheorem*{proof2}{Proof that $(1)\Rightarrow(2)$ from an anonymous reviewer}
\newtheorem*{proof3}{Alternative proof that $(1)\Rightarrow m $ is divisible by $k$}
\newtheorem*{introduction}{Introduction}
\newtheorem*{case}{Special case of the theorem}

\newcommand{\lcm}{\mathop\text{\rm lcm}}
\newcommand{\ord}{\mathop\text{\rm ord}}

\begin{document}

\maketitle

\begin{introduction}
In this article, I present a theorem determining a criterion for divisibility of two generalized Mersenne numbers, which are repunits of the same length in base-$a^m$ and base-$a^k$. In addition to the  general proof, I present an alternative proof for a special case of the theorem.

In this paper $d, m, k, a$ are any positive integers, $d > 1, a>1$.

There are several interpretations of the term "generalized Mersenne number". 

In this work, I define generalized Mersenne number $M_d(x) := \sum_{j=0}^{d-1}x^j$.

\smallskip
\textbf{Acknowledgments.} The author thanks Prof. N. Osipov, A. Kalmynin and Prof. A. Skopenkov for their continued support throughout the different stages of the project.

\end{introduction}

\begin{theorem}
The following conditions are equivalent.\\ (1) The number $M_d(a^m)$ is divisible by the number $M_d(a^k)$. \\ (2) $m$ is divisible by $k$ and $\gcd(\frac{m}{k};d)=1$.
\end{theorem}

\begin{proof}
Define
\[
Q:=\frac{M_d(a^m)}{M_d(a^k)}
\]

\smallskip
\emph{Proof that $(1)\Leftarrow(2)$.} Suppose that $m$ is divisible by $k$ and $\gcd{(\frac{m}{k},d)}=1$. Let $n=\frac{m}{k}$ so that $\gcd{(n,d)}=1$. The corresponding quotient can be written as
\[
Q=\frac{(a^{md}-1)(a^k-1)}{(a^m-1)(a^{kd}-1)}=
\frac{(b^{nd}-1)(b-1)}{(b^n-1)(b^d-1)}
\]
where $b=a^k \ge 2$. We have $\gcd{(b^n-1,b^d-1)}=b^{\gcd{(n,d)}}-1=b-1$. Thus,
\[
Q=\frac{b^{nd}-1}{\lcm(b^n-1,b^d-1)}
\]
and $Q$ is an integer since $b^{nd}-1$ is a common multiple of $b^n-1$ and $b^d-1$.

\smallskip
\emph{Proof that $(1)\Rightarrow(2)$.} Suppose the quotient
\[
Q=\frac{(a^{md}-1)(a^k-1)}{(a^m-1)(a^{kd}-1)}
\]
to be an integer. We now prove that $m$ is divisible by $k$ and that $\gcd(\frac{m}{k};d)=1$ separately.

\smallskip
\emph{Proof that (1) $\Rightarrow$ $m$ is divisible by $k$.} Using the so-called \emph{Zsigmondy's theorem} (see \cite[Theorem V]{Birkhoff&Vandiver-1904} or the original Zsigmondy's paper \cite{Zsigmondy-1892}), show that $m$ is divisible by $k$.

\smallskip
\emph{Case when $kd=2$.} If $kd=2$ then $k=1$ and the claim is trivially true.

\smallskip
\emph{Case when $kd=6$ and $a=2$.} If $kd=6$ and $a=2$, then either $k=1$ (then $m$ is divisible by $k$) or $(k,d) \in \{(2,3),(3,2)\}$. In the case $(k,d)=(2,3)$ the quotient
\[
Q=\frac{2^{2m}+2^m+1}{21}
\]
is an integer only for $m=2n$ with $n$ not divisible by $3$. In the case $(k,d)=(3,2)$ the quotient
\[
Q=\frac{2^m+1}{9}
\]
is an integer only for $m=3n$ with $n$ odd. Consequently, in both cases $m$ is divisible by $k$.

\smallskip
\emph{All other cases.} In remaining cases, by Zsigmondy's theorem there exists a prime divisor $p$ of $a^{kd}-1$ such that $\ord_p{(a)}=kd$. Due to $kd>k$ we conclude that $a^k \not\equiv 1 \pmod{p}$. Since $Q$ is an integer, we obtain $a^{md} \equiv 1 \pmod{p}$. Hence $md$ is divisible by $kd$ or, equivalently, $m$ is divisible by $k$.

\smallskip
\emph{Proof that (1) $\Rightarrow$ $\gcd(\frac{m}{k};d)=1$.} Letting $n=\frac{m}{k}$ again, we get
\[
Q=\frac{M_d(b^n)}{M_d(b)}
\]
with $b=a^k \ge 2$. Assuming $Q$ to be an integer, we need to prove that $\gcd{(n,d)}=1$.

Let $\gcd{(n,d)}=l$ and $n_1=n/l$, $d_1=d/l$. Then
\begin{equation}
\label{eq-1}
M_d(b^n) \equiv lM_{d_1}(b^l) \pmod{M_d(b)}
\end{equation}
as we will see later. Also, we have $M_d(b)=M_{d_1}(b^l)M_l(b)$ so that
\[
0<lM_{d_1}(b^l)=\frac{l}{M_l(b)}M_d(b)<M_d(b)
\]
for all $b \ge 2$ and $l>1$. Thus, $M_d(b^n)$ cannot be divisible by $M_d(b)$ for $l>1$. Therefore, $l=1$.

It remains to prove (1). We use the following well-known fact: for any $N$ divisible by $d$, the congruence $b^N \equiv 1 \pmod{M_d(b)}$ holds. We have
\[
M_d(b^n)=\sum_{i=0}^{d-1}b^{ni}=\sum_{q=0}^{l-1}\sum_{r=0}^{d_1-1}b^{ln_1(d_1q+r)} \equiv 
\sum_{q=0}^{l-1}\sum_{r=0}^{d_1-1}b^{ln_1r}=l\sum_{r=0}^{d_1-1}b^{ln_1r}=\\
\]
\[
=l\sum_{r=0}^{d_1-1}b^{l(j_r+d_1q_r)} \equiv l\sum_{r=0}^{d_1-1}b^{lj_r}=
l\sum_{r=0}^{d_1-1}b^{lr}=lM_{d_1}(b^l) \pmod{M_d(b)}.
\]
Here $n_1r=j_r+d_1q_r$ for some integers $j_r, q_r$ ($0 \le j_r<d_1$) such that
\[
\{j_0,j_1,\dots,j_{d_1-1}\}=\{0,1,\dots,d_1-1\}
\]
due to $\gcd{(n_1,d_1)}=1$.

In the ring $\mathbb{Z}[x]$, the congruence (1) rewritten as
\[
M_d(x^n) \equiv lM_{d_1}(x^l) \pmod{M_d(x)}
\]
holds too (and follows the same proof).
\end{proof}

\begin{remark}

\hspace*{\parindent} 1. $M_d(x^m)$ is divisible by $M_d(x^k)$ in the ring of polynomials $\mathbb{Z}[x]$ if and only if $m$, $k$, $d$ satisfy the same condition. The proof of the claim ``if'' almost exactly follows one given in part I. At the same time, the proof of the claim ``only if'' is simplified by considering the complex roots of polynomials with their multiplicities instead of Zsigmondy's theorem.

2. Proof that (1) $\Rightarrow$ $\gcd(\frac{m}{k};d)=1$ can be performed via Zsigmondy's theorem and the well-known LTE-lemma (see \cite{LTE}). Rewrite $Q=M_d(b^n)/M_d(b)$ in the form
\[
Q=\frac{(b^{nd}-1)(b-1)}{(b^n-1)(b^d-1)}.
\]
Assume $Q$ to be an integer while $\gcd{(n,d)}>1$. Let $q$ be a common prime divisor of $n$ and $d$. Let $n_1=n/q$ and $d_1=d/q$.

\smallskip
\emph{Case when $q \ge 3$.} By Zsigmondy's theorem, there is a prime divisor $p$ of $b^q-1$ such that $\ord_p{(b)}=q$. Then $b \not\equiv 1 \pmod{p}$. As a corollary we get $p \neq q$ and $p \neq 2$. Let $c=b^q$. We have
\[
\nu_p((b^n-1)(b^d-1))=\nu_p((c^{n_1}-1)(c^{d_1}-1))=
\]
\[
=\nu_p(c^{n_1}-1)+\nu_p(c^{d_1}-1)=2\nu_p(c-1)+\nu_p(n_1)+\nu_p(d_1).
\]
while
\[
\nu_p((b^{nd}-1)(b-1))=\nu_p(b^{nd}-1)=\nu_p(c^{qn_1d_1}-1)=
\]
\[
=\nu_p(c-1)+\nu_p(qn_1d_1)=\nu_p(c-1)+\nu_p(n_1)+\nu_p(d_1).
\]
Thus, $\nu_p((b^n-1)(b^d-1))>\nu_p((b^{nd}-1)(b-1))$ and $Q$ cannot be an integer. Contradiction.

\smallskip
\emph{Case when $q=2$ and $b+1 \neq 2^t$.} In this case $b+1$ has a prime divisor $p \neq 2$. Clearly, $b \not\equiv 1 \pmod{p}$. We can get a contradiction in the same way as above.

\smallskip
\emph{Case when $q=2$ and $b=2^t-1$ where $t \geqslant 2$.} Here, we have $n=2n_1$, $d=2d_1$. Then
\[
\nu_2(b^n-1)=\nu_2(b^2-1)-1+\nu_2(n)=t+\nu_2(n)=t+1+\nu_2(n_1)
\]
and similarly
\[
\nu_2(b^d-1)=t+1+\nu_2(d_1), \quad \nu_2(b^{nd}-1)=t+2+\nu_2(n_1)+\nu_2(d_1).
\]
Thus, we obtain
\[
\nu_2((b^n-1)(b^d-1))=2t+2+\nu_2(n_1)+\nu_2(d_1)>
\]
\[
>t+3+\nu_2(n_1)+\nu_2(d_1)=\nu_2((b^{nd}-1)(b-1))
\]
and $Q$ cannot be an integer again. Contradiction.

Hence, $\gcd(n, d) = 1$.
\end{remark}

\begin{proof2}
\emph{Proof that (1) $\Rightarrow$ $m$ is divisible by $k$.} Suppose that the quotient
\[
Q'=\frac{(a^{md}-1)(a^k-1)}{a^{kd}-1}
\]
is an integer. Then $m$ must be divisible by $k$. Indeed, let
\begin{gather*}
D=\gcd{(a^{md}-1,a^{kd}-1)}=a^{d\gcd{(m,k)}}-1,\\
a^{md}-1=DA', \quad a^{kd}-1=DA''
\end{gather*}
for some $A'$ and $A''$ such that $\gcd{(A',A'')}=1$. Then $DA'(a^k-1)=DA''Q'$. Hence,
\[
A'(a^k-1)=A''Q'.
\]
Then $A'(a^k-1)$ is divisible by $A''$.
Hence $A''$ is a divisor of $a^k-1$ because $\gcd{(A',A'')}=1$.
Thus, $DA''$ is a divisor of $D(a^k-1)$, i.e., $a^{kd}-1$ is a divisor of $(a^{d\gcd{(m,k)}}-1)(a^k-1)$. In particular,
\[
a^{kd}-1 \leqslant (a^{d\gcd{(m,k)}}-1)(a^k-1)<a^{d\gcd{(m,k)}+k}-1
\]
that yields $kd<d\gcd{(m,k)}+k$ or
\[
\frac{k}{\gcd{(m,k)}}<\frac{d}{d-1} \leqslant 2
\]
(recall that $d \geqslant 2$). Then $k=\gcd{(m,k)}$ so that $m$ is divisible by $k$.

\smallskip
\emph{Proof that (1) $\Rightarrow$ $\gcd(\frac{m}{k};d)=1$.} Let $n=m/k$ and assume the number
\[
\frac{(b^{nd}-1)(b-1)}{(b^n-1)(b^d-1)}=\frac{b^{nd}-1}{b^n-1}:\frac{b^d-1}{b-1}=
\frac{b^{nd}-1}{b^d-1}:\frac{b^n-1}{b-1}
\]
to be an integer (where $b=a^k \geqslant 2$). Let us introduce
\[
M:=\frac{b^{\gcd{(n,d)}}-1}{b-1}.
\]
Clearly, we have
\[
\frac{b^d-1}{b-1} \equiv \frac{b^n-1}{b-1} \equiv 0 \pmod{M}.
\]
As a corollary of our assumption, we get
\begin{gather*}
0 \equiv \frac{b^{nd}-1}{b^n-1}=1+b^n+\ldots+b^{n(d-1)} \equiv d \pmod{M},\\
0 \equiv \frac{b^{nd}-1}{b^d-1}=1+b^d+\ldots+b^{d(n-1)} \equiv n \pmod{M}.
\end{gather*}
Therefore, $\gcd{(n,d)}$ is divisible by $M$. In the case $\gcd{(n,d)}>1$ we obtain
\[
\gcd{(n,d)}<2^{\gcd{(n,d)}}-1 \leqslant M
\]
which is a contradiction. Thus, $\gcd{(n,d)}=1$ as desired.
\end{proof2}

\begin{proof3}
Define $g = \ord_{M_d(a^k)}{(a)}$.
We now prove that $g = kd$. Indeed, $a^{kd}-1$ is divisible by $M_d(a^k)$, so $kd$ is divisible by $g$. On the other hand, $M_d(a^k) > a^{k(d-1)}$, so $g > k(d-1)$. Thus, $g = kd$. Hence, (1) $\Rightarrow a^{md} - 1$ is divisible by $M_d(a^k) \Rightarrow m$ is divisible by $k$.
\end{proof3}

\begin{case}
\textbf{Proposition.}
\emph{The following conditions are equivalent.\\ (1) The number $M_d(n^m)$ is divisible by the number $M_d(n^k)$ for any integer $n$. \\ (2) $m$ is divisible by $k$ and $\gcd(\frac{m}{k};d)=1$.}

\smallskip
This result could be known. 
I learned it from \cite{Sy}. 
I. Syukrin presented an idea of a technical proof using canonical decomposition of numbers. 
Here I present a simple proof using roots of polynomials. 

(3) \emph{The polynomial $M_d(x^m)$ is divisible in $\mathbb{Z}[x]$ by the polynomial $M_d(x^k)$.}

Equivalence of (3) and (1) follows from \cite[Problem 3.4.2.b]{Sk} (there the assumption that $P$ does not have roots is superfluous).

\smallskip 
We now prove that both (3) and (2) are equivalent to the following statement.  
 
(4) \emph{The polynomial $(x^{md}-1)(x^k-1)$ is divisible in $\mathbb{Z}[x]$ by the polynomial $(x^m-1)(x^{kd}-1)$.}

We prove the equivalences 
$$(2)\Leftrightarrow(4)\Leftrightarrow(3).$$ 
Equivalence of (3) and (4) is obvious.

\smallskip
\emph{Proof that $(2)\Rightarrow(4)$.} 
This statement follows the same proof as the proof that (1) $\Leftarrow$ (2) in the Theorem.

\smallskip
\emph{Proof that $(2)\Leftarrow(4)$.} 
Since $(x^{md}-1)(x^k-1)$ is divisible in $\mathbb{Z}[x]$ by the polynomial $(x^m-1)(x^{kd}-1)$ and the value of $(x^m-1)(x^{kd}-1)$ for $x_1:=e^{2\pi i /kd}$ is 0, $e^{2\pi i /kd}$ is a root of both the numerator and denominator.
Thus, $(x_1^{md}-1) = 0$, because $(x_1^k-1)\not=0$. 
Hence, $md$ is divisible by $kd$. 
So, $m$ is divisible by $k$.

\smallskip
If $\gcd(u, d) > 1$, then $g(y) := (y^u - 1)(y^d - 1)$ has multiple roots other than 1, while $f(y) := (y^{ud} - 1)(y - 1)$ has no such roots. Thus, $f(y)$ is not divisible by $g(y)$ in $\mathbb{Z}[y]$. Then, by the following well-known assertion, $f(x^k)$ is not divisible by $g(x^k)$. 
Hence, $\gcd (u, d) = 1$.

So (2) holds.

\smallskip
\textbf{Assertion.}
\emph{Let $k$ be an integer and and $f(y), g(y) \in \mathbb{Z}[y]$. If $f(x^k)$ is divisible by $g(x^k)$ in $\mathbb{Z}[x]$, then $f(y)$ is divisible by $g(y)$ in $\mathbb{Z}[y]$.}
\end{case}


\begin{thebibliography}{MMM+}

\bibitem[Sk]{Sk} \emph{A. Skopenkov}, Mathematics Via Problems: Algebra.
\url{https://old.mccme.ru//circles//oim/algebra_eng.pdf}

\bibitem[Sy]{Sy} \emph{I. Syukrin}, Divisibility of generalized Mersenne numbers. \url{https://old.mccme.ru//circles//oim/mmks/works2024/syukrini5.pdf}. 

\bibitem[Bi]{Birkhoff&Vandiver-1904}
G.\,D.~Birkhoff and H.\,S.~Vandiver, On the Integral Divisors of $a^n-b^n$, \emph{Annals of Mathematics} \textbf{5} (1904), 173--180.

\bibitem[Zs]{Zsigmondy-1892}
K.~Zsigmondy, Zur Theorie der Potenzreste, \emph{Monatsh. f. Mathematik und Physik} \textbf{3} (1892), 265--284.

\bibitem[LTE]{LTE}
É. Lucas, "Théorie des Fonctions Numériques Simplement Périodiques. [Continued]", American Journal of Mathematics, vol. 1 (1878), pp. 197–240. \url{https://www.jstor.org/stable/2369308}

\bibitem[AOP]{AOP} 
\url{https://planetmath.org/AllOnePolynomial}

\bibitem[Re]{Re} 
\url{https://mathworld.wolfram.com/Repunit.html}

\end{thebibliography}
\end{document}